\documentclass[12pt,a4paper]{article}

\usepackage{bbm}
\usepackage[colorlinks]{hyperref}
\usepackage{amsfonts}
\usepackage{mathrsfs}
\usepackage{amsmath}
\usepackage{amsthm}
\usepackage{amssymb}
\usepackage{url}

\renewcommand{\mid}{:}
\newcommand{\Lip}{\operatorname{Lip}}
\newcommand{\sde}{\textsc{sde}}
\newcommand{\spde}{\textsc{spde}}
\newcommand{\Ord}[1]{\ensuremath{\mathcal O\big(#1\big)}}

\newcommand{\D}[2]{\mathchoice
  {\frac{\partial #2}{\partial #1}}
  {{\partial #2}/{\partial #1}}
  {{\partial #2}/{\partial #1}}
  {{\partial #2}/{\partial #1}}
  }
\newcommand{\Dn}[3]{\mathchoice
  {\frac{\partial^#2 #3}{\partial #1^#2}}%
  {{\partial^#2 #3}/{\partial #1^#2}}%
  {{\partial^#2 #3}/{\partial #1^#2}}%
  {{\partial^#2 #3}/{\partial #1^#2}}%
  }
\newcommand{\DD}[2]{\Dn{#1}2{#2}}

\newcommand{\cK}{\mathcal K_a}
\newcommand{\cKo}{\mathcal K_0}

\newtheorem{theorem}{Theorem}
\newtheorem{lemma}[theorem]{Lemma}
\newtheorem{corollary}[theorem]{Corollary}
\newtheorem{definition}[theorem]{Definition}
\theoremstyle{remark}
\newtheorem{remark}{Remark}

\topsep=\parskip

\title{Center manifolds for stochastic evolution equations}

\author{Xiaopeng Chen
\thanks{Beijing International Center for Mathematical Research, Peking University, Beijing, \textsc{China}. \protect\url{mailto:
chenxiao002214336@yahoo.cn} }
\and
Anthony J. Roberts\thanks{School of Mathematical Sciences, University of Adelaide, Adelaide, \textsc{Australia}.
\protect\url{mailto:anthony.roberts@adelaide.edu.au}}
\and
Jinqiao Duan\thanks{Institute for Pure and Applied Mathematics (IPAM), University of California, Los Angeles, USA \& Department of Applied Mathematics, Illinois Institute of Technology,   Chicago, IL 60616, USA.
\protect\url{mailto:jduan@ipam.ucla.edu}}
}

\date{\today}

\begin{document}
\maketitle

\begin{abstract}
Stochastic invariant manifolds are crucial in modelling  the dynamical behavior of    dynamical systems under uncertainty.
Under the assumption of exponential trichotomy,  existence and smoothness of center manifolds for a class of  stochastic evolution equations with linearly multiplicative noise are proved.
The exponential attraction and approximation to center manifolds are  also discussed.
\end{abstract}

\paragraph{Mathematics Subject Classifications (2010)} Primary 60H15, 35R60; Secondary 37D10,
34D35.

\paragraph{Keywords} Stochastic partial differential equations (\textsc{spde}s), exponential trichotomy,
center manifolds, stability,
dynamical approximation.

\tableofcontents

\section{Introduction}
Invariant manifolds are some of the most important invariant sets in nonlinear dynamical systems.
Stable, unstable, center and inertial manifolds have been widely studied in deterministic  systems.
But in many applications  the nonlinear  dynamical  systems are  influenced by  noise.
Recently invariant manifolds for stochastic differential equations and stochastic partial differential equations have been  explored.
Mohammed and Scheutzow~\cite{MohScheu99} focused
 on the existence of local  stable and unstable  manifolds for stochastic differential equations driven by semimartingales.
Boxler~\cite{Box}  obtained a stochastic
 version of   center manifold theorems for  finite dimensional random dynamical systems by using the multiplicative ergodic theorem.
Arnold~\cite{Arn98}  summarised various invariant manifolds  on  finite dimensional random dynamical systems.
Roberts~\cite{Rob1}   assumed  existence of  stochastic slow manifolds  for infinite stochastic partial differential equations in explorring  the  interactions of microscale noises and their macroscale modelling.
An unresolved problem is  the existence and nature of stochastic  invariant  manifolds  on infinite dimensional spaces.

Recently there are some theory developments about the problem.
Duan and others~\cite{Duan1, Duan2}  presented stable and unstable invariant manifolds for a class of stochastic partial differential
 equations under the assumption of exponential dichotomy or pseudo exponential dichotomy.  
Inertial manifolds are  generalization of center-unstable manifolds on   finite dimensional  spaces to  infinite dimensional spaces.
 Stochastic inertial manifolds on infinite dimensional spaces are constructed  by different methods~\cite{Ben, Car1, Prato}.
Center manifolds  are  important invariant manifolds.
It  has been  investigated extensively  for  infinite dimensional deterministic systems~\cite{Bat, Carr, Van}.
But there are only few papers  dealing  specifically with infinite dimensional center manifolds.
Gallay~\cite{Gal}  proved   the existence of  infinite dimensional center-stable manifolds and local center manifolds  for a class of  deterministic evolution equations  in Banach spaces.

This article establishes existence and properties of stochastic center manifolds  on infinite dimensional Hilbert spaces with multiplicative white noise in the following class of stochastic evolution equation
\begin{equation}\label{eq(2.1)}
    \frac{d u}{dt}=A u+F( u)+u\circ \dot{W}(t),
\end{equation}
where $u\in H$ is a Hilbert space typically defined on some spatial domain,  ${W}(t)$~is the standard $\mathbb{R}$-valued Wiener process on a probability space~$(\Omega,\mathcal{F}, \mathbb{P} )$,  which is only	dependent on  time.
A simple example is the linear stochastic parabolic equation
\begin{equation}\label{pa}
 u_{t}= u_{xx}+4u+u\circ\dot{W}(t), \quad u(0,t)=u(\pi,t)=0\,, \quad  t\geq 0\,.
\end{equation}
In this example the Hilbert space~$H$ is the function space~$L^2(0,\pi)$.
This example  stochastic partial differential equation (\textsc{spde}) has all three subspaces of interest:  stable,  center  and unstable subspaces.
Thus, the theory of Duan et al.~\cite{Duan1,Duan2} and Caraballo et al.~\cite{Car2} does not apply.
Sections~\ref{Sec2}--\ref{Sec4} establish the theory of center manifolds of such  \textsc{spde}s.
That is we prove the existence and smoothness of the center manifolds for  stochastic evolution equation~\eqref{eq(2.1)}.
The results generalize the work of Duan  et al.~\cite{Duan1, Duan2}, who proved the existence and smoothness of center-stable and center-unstable  manifolds.
The stochastic center manifolds we obtain  are generally  infinite dimensional.
Then  Section~\ref{Sec5} and~\ref{Sec6} proves the principle of  exponential attraction  from the stochastic center manifolds  and an approximation to the stochastic center manifolds.

Our main assumption is that there exists a spectral gap for the linear operator~$A$.
Such  a gap  allows us to construct the infinite dimensional  center manifolds.
Exponential dichotomy is one of the basic assumptions  in nonlinear dynamical systems.
The assumption plays a central role in the study of stable and unstable manifolds.
But the assumption of an exponential dichotomy~\cite{Duan1, Duan2}  is not a sufficient condition for proving the existence of center manifolds.
For example,  a  generator of a strongly continuous semigroup  has some of its spectrum on the imaginary axis,  which implies  the existence of  center manifolds, but Pr\"{u}ss~\cite{Pr} showed there is no exponential dichotomy.
The pseudo-exponential dichotomy~\cite{Duan1,Pli} is also  not a sufficient assumption to prove the existence  of  center manifolds since in the finite dimension case  we need to ensure  the existence of zero real part of the  spectrum of a bounded operator.
The concept of exponential trichotomy is important for center manifold theory  in infinite dimensional dynamical systems and non-autonomous systems~\cite{Bar, Gal, Pli}.
The existence of exponential trichotomy means  that the space is split  into three subspaces: center subspace, unstable subspace and stable subspace.
In  finite dimensional space  the exponential trichotomy is satisfied automatically.
Center manifold theory is based on the assumption of exponential trichotomy~\cite{Pli}.
Section~\ref{Sec2}  introduces the exponential trichotomy and conjugated random evolution equations.
Using the Lyapunov--Perron method, Section~\ref{Sec3}  proves the existence of   stochastic center manifolds  under the assumption of exponential trichotomy.
We  consider  the existence of local center manifolds  via a cut-off technique.
When the stable subspace or the unstable subspace disappear, then our theorems reduce to  the results by Duan and others ~\cite{Duan1, Duan2}.

Deterministic center manifold theory is widely studied and further extended by others in various contexts~\cite[e.g.]{Adi, Bat, Carr,Chi,  Har, Hen,  Van,  Van2}.
 The  method of  proving the existence of  center manifolds   is to show that the  center manifold is the fixed point of a certain of operator.
It is also the  intersection of the center-unstable manifolds and center-stable manifolds~\cite{Bat,Chi1}.
In this  paper,   we define the stochastic center manifold as a graph of Lipschitz map.
We directly prove the existence of such map by contraction mapping theorem, which is the method of  Lyapunov--Perron~\cite{Van2}.
It  is different from the Hadamard's graph transform method~\cite{Box,Duan2}, which is more directly geometrical.
Section~\ref{Sec4} proves the  center manifolds  are smooth  by using the method of Lyapunov--Perron.

There are  many applications of the theory of center manifolds such as the dimensional reduction~\cite{Du, Xu},  bifurcation~\cite{Carr},  discretisation~\cite{Rob, Rob2}. 
The stability  of  an equilibrium point  on  center manifolds  has been considered in deterministic case~\cite{Adi, Pal}.
In the random case
we show the  equilibrium  points are   asymptotically stable  when the equilibrium points restricted  on  the stochastic center manifolds  are  asymptotically stable  (Section~\ref{Sec5}).
We then prove that there  exists  an approximation of the stochastic centre manifold (Section~\ref{Sec6}).
There has been  some recent results on the approximation of invariant manifolds.
Wang and Duan~\cite{Wang}  proved an asymptotic completeness property  for a class of stochastic partial differential equations.
Sun et al.~\cite{Sun} showed that  an invariant manifold is approximated by  a deterministic manifold when the noise is small.
Chen et al.~\cite{Chen} gave a geometric shape of invariant manifolds for a class of stochastic partial differential equations.
Bl\"{o}mker and Haire~\cite{Bl}  considered a different approach to approximate  stochastic partial differential equations by amplitude equations.
Section~\ref{Sec6} gives a  class of stochastic approximation to random evolution equations by stochastic  center manifolds, which make some progress in the computation of stochastic  center manifolds~\cite{Rob1}.
For this pathwise approximation  version, the difficulty is to find a  random variable to approximate the center manifolds.
This is different from the random norm~\cite{Box} when using Carr's result~\cite{Carr}.
Section~\ref{Sec7} uses  some examples to illustrate our results.

\section{Exponential trichotomy} \label{Sec2}

We  consider the stochastic
evolution equation~\eqref{eq(2.1)}.
We assume that the linear operator $A: D(A)\to H$ generates a
strongly continuous semigroup $S(t):=\exp({At})$ on a Hilbert space~$(H, |\cdot|)$, which satisfies the
\emph{exponential trichotomy } with exponents $\alpha >\gamma> 0>-\gamma>-\beta$ and bound~$K$.

First we introduce the exponential trichotomy,  which generalizes the exponential dichotomy~\cite{Cop, Duan1, Duan2, Sac}.
\begin{definition} \label{trich}
A strongly continuous semigroup $S(t):=\exp({At})$ on~$H$ is said to  satisfy  the exponential trichotomy  with exponents $\alpha >\gamma> 0>-\gamma>-\beta$\,, if there exist continuous projection operators~$P^c$, $P^u$ and~$P^s$ on~$H$ such
that the following four conditions hold.
\begin{enumerate}
\item $\operatorname{id}=P^c+P^s+P^u$,  where $P^iP^j=0$ for $ i\neq j$, $i,j\in\{c,s,u\}$, and  $\operatorname{id}$~is the identity operator.
\item   $P^c \exp (At)=\exp(At)P^c$, $P^u\exp(At)=\exp(At)P^u$, $P^s\exp(At)=\exp(At)P^s$ for $t\geq 0$\,.
\item  Denote the reducing subspaces $H^c:= P^cH$\,, $H^u:=P^uH$ and  $H^s:=P^sH$\,.
We  call~$H^c$,  $H^u$ and~$H^s$ the center subspace, the unstable
subspace  and the stable subspace, respectively.
The restriction~$\exp(At)|_{H^i}$\,,  $t\geq 0$ are
isomorphism from~$H^i$ onto~$H^i$,  and we
define~$\exp(At)|_{H^i}$ for $t <0$ as the inverse map of~$\exp(-At)|_{H^i}$, $i\in\{c,s,u\}$.
\item
\begin{eqnarray}
&&|\exp(At)P^c v|\leq K \exp({\gamma |t|})|P^cv|,  \quad t \in  \mathbb{R},   v\in H,   \label{trichotomy1} \\
&&|\exp(At)P^u v|\leq K \exp({\alpha t}) |P^uv|,\quad t\leq 0\,,  v\in H,   \label{trichotomy2}\\
&&|\exp(At) P^sv|\leq K \exp({-\beta t}) |P^sv|,  \quad t\geq 0\,,   v\in H. \label{trichotomy3}
\end{eqnarray}
\end{enumerate}
\end{definition}

\begin{remark}If $H$~is a  finite dimensional space  and there exist     eigenvalues with real part   zero, less than zero,  and greater than zero,  then we have the
exponential trichotomy.
If $H$~is an infinite dimensional space and  the spectrum of~$A$ satisfies $\sigma(A)=\sigma^s\cup\sigma^c\cup\sigma^u$, where $\sigma^s=\{\lambda\in \sigma(A)\mid \operatorname{Re} \lambda \leq-\beta\}$, $\sigma^c=\{\lambda\in \sigma(A)\mid| \operatorname{Re} \lambda| \leq \gamma\}$, $\sigma^u=\{\lambda\in \sigma(A)\mid \operatorname{Re} \lambda \geq \alpha\}$, 
and $A$~generates a strong continuous semigroup,  then we  have the exponential trichotomy~\cite{Gal}.
\end{remark}

 We assume the nonlinear term~$F$ satisfies $F(0)=0$\,,  and assume it to be Lipschitz
continuous on~$H$, that is,
\[
|F(u_1)-F(u_2)|\le \Lip  F\,|u_1-u_2 |
\]
with the sufficiently small  Lipschitz constant $\Lip  F>0$\,.
If  the  nonlinear term~$F$ is locally Lipschitz, let $F^{(R)}(u)=\chi_R(u)F(u)$, where $\chi_R(u)$~is a cut-off function, and then  $F^{(R)}$~is  globally Lipschitz with Lipschitz  constant~$R\Lip  F$.

Second  we use a coordinate transform   to
convert  the stochastic partial differential equation~\eqref{eq(2.1)}
into a random evolution equation~\cite{Car3, Duan1,Duan2}.
Consider the   linear stochastic
differential equation,
\begin{equation}\label{eq(2.4)}
dz+\mu z\,dt=dW,
\end{equation}
where $\mu$ is a positive parameter.
A solution of this equation is called an Ornstein--Uhlenbeck
process.
Let $C_0(\mathbb{R},\mathbb{R})$ be continuous functions on~$\mathbb{R}$, the associated distribution~$\mathbb{P}$ is a Wiener measure defined on the Borel-$\sigma$-algebra $\mathcal{B}(C_0(\mathbb{R},\mathbb{R}))$.
Define $\{\theta_t\}_{t\in\mathbb{R}}$ to be the metric dynamical system generated by the Wiener process~$W(t)$.
A random dynamical system on a metric space~$H$ on $(\Omega,\mathcal{B}, \mathbb{P}, \theta_t)$ is a measurable map
\begin{eqnarray*}\phi:\mathbb{R}\times H  \times  \Omega&\to& H\\
(t,u, \omega)& \mapsto &\phi(t,u, \omega)\end{eqnarray*}
such that
\begin{enumerate}
\item   $\phi(0,u,\omega)=\operatorname{id}$;
\item   $\phi(t+s,u,\omega)=\phi(t,u,\theta_s\omega)\phi(s,u,\omega)$,   for all $t,s \in \mathbb{R}$ for almost all $\omega \in \Omega$.
\end{enumerate}
 Caraballo et~al.~\cite{Car3} and Duan et al.~\cite{ Duan1} established the following lemma, which is used in the proof of the existence of  stochastic  invariant manifolds.
\begin{lemma}\label{lem(2.1)}
\begin{enumerate}
\item   There exists a $\{\theta_t\}_{t\in\mathbb{R}}$-invariant set
$\Omega\in\mathcal{B}(C_0(\mathbb{R},\mathbb{R}))$ of full measure
with sublinear growth:
\[
\lim_{t\to\pm\infty}\frac{|\omega(t)|}{|t|}=0\,,\quad
\omega\in\Omega\,.
\]
\item   For $\omega\in\Omega$ the random variable
$z(\omega)=-\mu\int_{-\infty}^0\exp(\mu\tau)\omega(\tau)\,d\tau$
exists and generates a unique stationary solution of the Ornstein--Uhlenbeck \sde~(\ref{eq(2.4)})
given by the convolutions
\begin{eqnarray*}
\Omega\times \mathbb{R}\ni(\omega,t)\to z(\theta_t\omega)
&=&-\mu\int_{-\infty}^0\exp(\mu\tau)\theta_t\omega(\tau)\,d\tau\\
&=&-\mu\int_{-\infty}^0\exp(\mu\tau)\omega(\tau+t)\,d\tau+\mu\omega(t).
\end{eqnarray*}
The mapping $t\mapsto z(\theta_t\omega)$
is continuous.
\item    In particular,
\begin{equation*}
\lim_{t\to\pm\infty}\frac{|z(\theta_t\omega)|}{|t|}=0\quad
   \text{for }\omega\in\Omega\,.
\end{equation*}
\item    In addition,
\begin{equation*}
\lim_{t\to\pm\infty}\frac{1}{t}\int_0^tz(\theta_\tau\omega)\,d\tau=0\quad
  \text{for }\omega\in\Omega\,.
 \end{equation*}
 \end{enumerate}
\end{lemma}

We now replace the Borel-$\sigma$-algebra  $\mathcal{B}(C_0(\mathbb{R},\mathbb{R}))$ by
\[
\mathcal{F}=\{\Omega\cap D \mid D\in
\mathcal{B}(C_0(\mathbb{R},\mathbb{R}))\}
\]
for $\Omega$ given in Lemma~\ref{lem(2.1)}.
The probability
measure is the restriction of the Wiener measure to this new
$\sigma$-algebra, which is  also denoted by~$\mathbb{P}$.
In the
following we consider a random dynamical system over  the metric dynamical system~$
(\Omega,\mathcal{F},\mathbb{P},\theta)$.

We show that  the solution of equation~(\ref{eq(2.1)}) define  a random dynamical system.
For any $u^*\in H$ and $\omega\in \Omega$\,,
we introduce the  coordinate transform
\begin{equation}\label{eq(2.7)}
u=T(u^*, \omega)=u^*\exp[{z(\theta_t\omega)}]
\end{equation}
and its inverse transform
\begin{equation}\label{eq(2.8)}
T^{-1}(u, \omega)=u\exp[{-z(\theta_t\omega)}].
\end{equation}

Without loss of generality, we drop the $*$.
We transform  the stochastic evolution equation~\eqref{eq(2.1)}  into the following partial differential equation with random coefficients
\begin{equation}\label{eq(2.5)}
\frac{du}{dt}=Au+z(\theta_t\omega)u+G(\theta_t\omega,u),\quad
u(0)=u_0\in H,
\end{equation}
where $G(\omega,u)=\exp[{-z(\omega)}]F(\exp[{z(\omega)}]u)$.
We have
\[
|G(\omega,u_1)-G(\omega,u_2)|\le \Lip_uG\,  |u_1-u_2 |,
\]
where $\Lip_u G$ denotes the Lipschitz constant of~$G(\cdot, u)$ with
respect to~$u$.
For  any $\omega\in\Omega$ the function~$G$ has the
same global Lipschitz constant  as~$F$ by the construction of~$G$. 
Define \begin{equation*}\Psi_A(t,s)=\exp\left[A(t-s)+\int_s^tz(\theta_r\omega)\,dr\right]\end{equation*} as a `state transition operator' for the linear random partial differential equation~\eqref{eq(2.5)} in the case
when the nonlinearity is neglected,  $G(\omega,u)=0$\,.
The solution is interpreted in a mild sense
\begin{equation*}\label{eq(2.6)}
u(t)=\Psi_A(t,0)u_0 +\int_0^t
\Psi_A(t,s)G(\theta_s\omega,u(s))\,ds\,.
\end{equation*}
This equation has  a unique measurable solution from the  Lipschitz continuity of~$G$ and measurable property of~$z(\omega)$.
Hence the solution mapping
$(t,v,\omega)\mapsto u(t, v, \omega)$
generates a random dynamical system. 

Duan et al.~\cite{Duan2} proved that  the  solution of~\eqref{eq(2.5)} is the coordinate transformation the  solution of~\eqref{eq(2.1)}.

\begin{lemma}\label{lem(2.2)}
Suppose that $u$~is the random dynamical system generated by~(\ref{eq(2.5)}).
Then
\begin{equation}\label{eq(2.9)}
(t, v,\omega)\mapsto T^{-1}(\theta_t\omega,
u(t,T(\omega,v),\omega))=:\hat u(t, v,\omega)
\end{equation}
is a random dynamical system.
For any $v\in H$ this process
$(t,\omega)\mapsto  \hat u(t, v,\omega)$ is a solution to~(\ref{eq(2.1)}).
\end{lemma}


\section{Stochastic center manifolds} \label{Sec3}

We introduce the definition of stochastic center manifolds.
A basic tool of proving the existence of  stochastic center manifolds is to define an appropriate function space, which is a Banach space.

\begin{definition}
A random set~$M(\omega)$ is called an (forward) invariant set for a random
dynamical system~$\phi(t, \omega, x)$ if
$\phi(t,M(\omega),\omega)\subset M(\theta_t\omega)$ for
$t\geq 0$\,.
If we can represent~$M(\omega)$ by a graph of a  (Lipschitz)
mapping from the center subspace to its complement,
$h^c(\cdot,\omega): H^c\to H^u\oplus H^s$,
such that
$M(\omega)=\{v + h^c(v,\omega) \mid v\in H^c\}$,
$h^c(0,\omega)=0$\,, and the tangency condition that  the derivative $Dh^c(0,\omega)=0$\,, $h^c(v,\cdot)$ is measurable for every $v\in H^c$,
then $M(\omega)$~is called a (Lipschitz) center
manifold, often denoted as~$M^c(\omega)$.  $M(\omega)$~is called a local center manifold if it is a graph of a (Lipschitz) mapping~$\chi_R(v)h^c(\cdot,\omega)$.
\end{definition}

We first show the existence of a Lipschitz
center manifold for the random partial differential
equation~\eqref{eq(2.5)}.
Then, we apply the inverse transformation~$T^{-1}$ to get a center
manifold for the stochastic evolution equation~(\ref{eq(2.1)}).

For each $\eta>0$\,,  we denote the Banach
space
\begin{align*}
C_\eta={}&\left \{ \phi\in C(\mathbb{R}, H) \mid  \sup_{t\in \mathbb{R}}\exp\left[{-\eta |t| - \int_0^t
z(\theta_r\omega) \,dr}\right] | \phi(t)|
 < \infty\right\}
\end{align*}
with the norm
\[\left|\phi\right|_{C_\eta}=\sup_{t\in \mathbb{R}}\exp\left[{-\eta |t|-\int_0^t z(\theta_r\omega)
\,dr}\right] |\phi(t)|.
\]
The set~$C_\eta$ is the set of  `slowly varying' functions.
We know that the functions  are controlled by
 $\exp\left[{\eta |t|+\int_0^t z(\theta_r\omega)
\,dr}\right]$.
Let
\[M^c(\omega)=\left\{u_0\in H \mid u(\cdot, u_0, \omega) \in C_\eta\right\},
\]
where  $u(t, u_0, \omega)$~is  the solution of~(\ref{eq(2.5)}) with the
initial data $u (0)= u_0$\,.

  We
prove that $M^c(\omega)$~is invariant and is  the
graph of a Lipschitz function.

For simple the expression of the proof, define the operator $B: \mathbb{R} \to L(H)$ by
\[ B(t):= \begin{cases}-\exp(At) P^u, & t \leq 0\,; \\   \exp(At) P^s, &  t \geq 0\,. \end{cases}\]
Different from the proof process of Duan et al.~\cite{Duan1,Duan2}, we need analysis the behavior of the  solution on center subspace.
\begin{theorem} \label{Thm(3.1)}    If   $\gamma < \eta<\min\{\beta, \alpha\}$ such that the nonlinearity term is sufficiently small,
\begin{eqnarray}K    \Lip_u G   \left(\frac{1}{\eta-\gamma}+\frac{1}{\beta-\eta}+\frac{1}{\alpha-\eta}\right)<
1\,,\label{pr1}\end{eqnarray} then there exists a  center manifold for
 the random partial differential  equation~(\ref{eq(2.5)}),  which is written as the graph
\[
M^c(\omega) = \{v+h^c(v,\omega)\mid v \in H^c\},
\]
where $h^c(\cdot,  \omega) : H^c\to  H^u\oplus H^s$ is a Lipschitz continuous mapping from the center subspace and
satisfies $h^c(0,\omega)=0$\,. 
\end{theorem}

\begin{proof}
 First we claim that $u_0 \in M^c(\omega)$ if and only if  there exists a slowly varying function $u(\cdot, u_0, \omega)\in
C_\eta$  with
\begin{align}\label{eq(3.2)}\begin{split}
u(t, u_0,\omega)={}&\Psi_A(t,0)v+\int_0^t
\Psi_A(t,s)P^cG(\theta_s\omega,u(s))\,ds
\\&{} +\int^{+\infty}_{-\infty}\Psi_0(t,s) B(t-s)G(\theta_s\omega,u(s))\,ds\,,
\end{split}
\end{align}
where $v = P^c u_0$, $\Psi_0(t,s)=\exp\left[\int_s^tz(\theta_r\omega)\,dr\right]$.  

To prove this claim,  first we let $u_0 \in M^c(\omega)$.
By
using the variation of constants formula, the solution on each subspace denoted as
\begin{align}&\label{eq(3.3)}
P^cu(t, u_0, \omega) = \Psi_A(t,0)
v + \int_0^t\Psi_A(t,s)
P^cG(\theta_s\omega,u)\,ds\,.
\\&
\label{eq(3.4)}
P^u u(t,u_0, \omega) =\Psi_A(t,\tau)P^uu(\tau, u_0, \omega)+\int_{\tau}^t \Psi_A(t,s)
P^uG(\theta_s\omega,u)\,ds\,.
\\&\label{eq(3.44)}
P^s u(t, u_0, \omega) =\Psi_A(t,\tau)P^s u(\tau, u_0, \omega)+\int_{\tau}^t \Psi_A(t,s)
P^s G(\theta_s\omega,u)\,ds\,.
\end{align}
Since the slowly varying function $u \in C_\eta$\,, we have for $t < \tau $ that the magnitude
\begin{align*}
\left|\Psi_A(t,\tau)P^uu(\tau, u_0, \omega)\right|
&\leq K\exp[{\alpha(t-\tau)} ]\Psi_0(t,0)\exp ({\eta
\tau} )|u|_{C_\eta} \\
& = K\Psi_\alpha(t,0)
\exp[{-(\alpha-\eta)\tau}] |u|_{C_\eta}
\\& \to 0 \quad\text{as }\tau \to +\infty.
\end{align*}
For  $t > \tau $\,,
\begin{align*}
\left|\Psi_A(t,\tau)P^su(\tau, u_0,
\omega)\right|
&\leq K\exp[{-\beta(t-\tau)}]\Psi_0(t,0)\exp({\eta
\tau}) |u|_{C_\eta} \\
& =K\Psi_{-\beta}(t,0)
\exp[{(\beta+\eta)\tau}] |u|_{C_\eta}
\\&\to 0 \quad\text{as }\tau \to -\infty.
\end{align*}
Then, taking the two separate limits $\tau \to \pm  \infty$   in~(\ref{eq(3.4)}) and~\eqref{eq(3.44)} respectively,
\begin{align}&\label{eq(3.5)}
P^uu(t, u_0, \omega)=\int^t_\infty
\Psi_A(t,s)
P^uG(\theta_s\omega,u(s))\,ds\,,
\\&\label{eq(3.55)}
P^su(t, u_0, \omega)=\int^t_{-\infty}
\Psi_A(t,s)
P^sG(\theta_s\omega,u(s))\,ds\,.
\end{align}
Combining~(\ref{eq(3.3)}),  (\ref{eq(3.5)}) and~(\ref{eq(3.55)}), we have~(\ref{eq(3.2)}).
The converse  follows from a direct computation.

Next we prove  that for any given $v \in H^c$, the centre subspace, the integral
equation~(\ref{eq(3.2)}) has a unique solution in the slowly varying functions space~$C_\eta$.
Let
\begin{align}\label{eq(3.22)}\begin{split}
J^c(u,v):={}&\Psi_A(t,0)v+\int_0^t
\Psi_A(t,s)
P^cG(\theta_s\omega,u(s))\,ds\\&{}
+\int^{+\infty}_{-\infty}\Psi_0(t,s)B(t-s)G(\theta_s\omega,u(s))\,ds\,.
\end{split}
\end{align}
$ J^c$~is
well-defined from $C_\eta \times H^c$ to the slowly varying functions space~$C_\eta$.
For each pair of slowly varying functions
$u, \bar u \in C_\eta$\,, we have  that for $\gamma < \eta<\min\{\beta, \alpha\}$,
\begin{align}\label{eq(3.6)}
\begin{split}
&| J^c(u,v ) -  J^c (\bar u, v ) |_{C_\eta} \\
&\leq\sup_{t\in \mathbb{R}} \Bigg\{\exp\left[
{-\eta |t| - \int_0^t
z(\theta_s\omega) \,ds}\right] \bigg |\int_0^t
\Psi_A(t,s)
P^c(G(\theta_s\omega,u)-G(\theta_s\omega,\bar u))\,ds
\\&
\quad \quad{}
+\int^{+\infty}_{-\infty}\Psi_0(t,s)B(t-s)(G(\theta_s\omega,u)-G(\theta_s\omega,\bar u))\,ds\bigg|
\Bigg\}\\
&\leq\sup_{t\in \mathbb{R}} \Bigg\{K \Lip_u G |u-\bar
u|_{C_\eta}\bigg|\int_0^{t} \exp[{(\gamma-\eta)|t-s|}]\,ds
\\&  \quad \quad{}
+\int^t_{-\infty}\exp[{(\eta-\beta)(t-s)}]\,ds
+\int_t^{+\infty}\exp[{(\alpha-\eta)(t-s)}]\,ds\bigg |
\Bigg\}\\
&\leq K \Lip_u
G\left(\frac{1}{\eta-\gamma}+\frac{1}{\beta-\eta}+\frac{1}{\alpha-\eta}\right) |u-\bar
u|_{C_\eta}.
\end{split}
\end{align}

From equation~\eqref{eq(3.6)}, $ J^c$~is Lipschitz continuous in~$v$.
By the theorem's precondition~\eqref{pr1},
$ J^c$~is a uniform contraction with respect to the parameter~$v$.
By the uniform contraction mapping principle,
for each $v \in H^c$,  the mapping $J^c (\cdot , v )$ has a
unique fixed point $u(\cdot,v, \omega) \in C_\eta$\,.
Combining equation~\eqref{eq(3.22)} and equation~\eqref{eq(3.6)},
\begin{align}\label{eq(3.7)}
|u(\cdot ,  v,   \omega) - u (\cdot,  \bar v, \omega
)|_{C_\eta}\leq \frac{K}{1-K \Lip_u
G\left(\frac{1}{\eta-\gamma}+\frac{1}{\beta-\eta}+\frac{1}{\alpha-\eta}\right)}|v - \bar v |,
\end{align}
for each
 fixed point~$u(\cdot,v, \omega)$.
 Then for each time $t\geq 0$\,,
$u(t,   \cdot, \omega )$~is Lipschitz from the center subspace~$H^c$ to slowly varying  functions~$C_\eta$.
 $u(\cdot,  v, \omega  )\in C_\eta$ is a unique solution
of the integral equation~(\ref{eq(3.2)}).
Since $u(\cdot , v,   \omega)$ can be an $\omega$-wise limit of the
iteration of contraction mapping~$J^c$ starting at~$0$ and $J^c$
maps a $\mathcal{F}$-measurable function to a $\mathcal{F}$-measurable function,
$u(\cdot, v,   \omega)$~is $\mathcal{F}$-measurable.
Combining $u(\cdot, v, \omega)$~is  continuous with respect to~$H$, we have $u(\cdot , v,   \omega)$ is measurable with respect to~$(\cdot, v, \omega)$.

Let $h^c (v,\omega) := P^s u (0, v, \omega)\oplus P^u u (0, v, \omega)$.
Then
\[
h^c(v, \omega) =\int^{+\infty}_{-\infty} \Psi_0(0,s)
B(-s)G(\theta_s\omega,u(s,v,\omega))\,ds\,.
\]
We see that~$h^c$ is $\mathcal{F}$-measurable and  $h^c(0,\omega)=0$\,.
 From the definition of~$h^c(v,
\omega)$ and the claim that $ u_0 \in M^c(\omega)$  if and only if
there exists $u(\cdot, u_0, \omega)~\in~C_\eta$ with $u(0)=u_0$ and satisfies~(\ref{eq(3.2)}) it follows that $ u_0 \in M^c(\omega)$ if and only
if there exists $v \in H^c$ such that $ u_0=v+h^c(v,
\omega)$, therefore,
\begin{equation*}
M^c(\omega) = \{v+ h^c(v, \omega) \mid v \in H^c\}.
\end{equation*}

Next we prove
that for any $x\in H$\,, the function
\begin{equation}\label{eq(3.8)}
\omega\to\inf_{y\in H^c}\left |x-(y+ h^c(y, \omega))\right |
\end{equation}
is measurable.
Let $H'$~be a countable dense subset of the separable space~$H$.
From  the continuity of~$h^c(\cdot,
\omega)$,
\begin{equation}\label{eq(3.9)}
\inf_{y\in H^c}|x-(y+ h^c(y, \omega))|
=\inf_{y\in H'}|x-P^cy-h^c(P^cy, \omega)|.
\end{equation}
The measurability of~(\ref{eq(3.8)}) follows since
$\omega\to h^c(P^cy,\omega)$ is measurable for any $y\in H'$.

Finally, we show that $M^c(\omega)$~is invariant, that is for each
$u_0\in M^c(\omega)$, $u(s, u_0, \omega) \in M^c(\theta_s\omega)$
for all $ s\geq 0$\,.
Since for $s \geq 0$\,,
$u(t+s, u_0,\omega)$ is a solution of
\[
\frac{du}{dt}=Au+z(\theta_t(\theta_s\omega))u+G(\theta_t(\theta_s\omega),u),\quad
u(0)=u(s, u_0, \omega).
\]
Thus $u(t,u(s,u_0,\omega), \theta_s\omega)=u(t+s, u_0,\omega)$ and  $u(t,u(s,u_0,\omega),
\theta_s\omega) \in C_\eta$\,.
So we have  $u(s, u_0, \omega) \in
M^c(\theta_s\omega)$

\end{proof}

\begin{theorem} \label{Thm(3.2)} Suppose that $u$~is the solution of random partial differential equation~(\ref{eq(2.5)}), then
$\widetilde {M}^c(\omega)=T^{-1}(\omega,M^c(\omega))$ is a
center manifold of the stochastic partial differential equation~(\ref{eq(2.1)}).
\end{theorem}
\begin{proof}
Denote~$\tilde u(t, \omega, x)$ the solution of~(\ref{eq(2.1)}).
From Lemma~\ref{lem(2.2)},  for $t\geq 0$\,,
\begin{align*}
\tilde u(t,\omega,&\widetilde  {M}^c(\omega))=T^{-1}(\theta_t\omega,u(t,\omega,T(\omega,\widetilde {M}^c(\omega))))\\
&=T^{-1}(\theta_t\omega, u(t,\omega,M^c(\omega)))\subset
T^{-1}(\theta_t\omega,M^c(\theta_t\omega)) =\widetilde
{M}^c(\theta_t\omega).
\end{align*}
So $\widetilde {M}^c(\omega)$ is an invariant set.
Note
that
\begin{align*}
\widetilde {M}^c(\omega) &= T^{-1}(\omega,M^c(\omega))\\ &
=\big\{u_0=T^{-1}(\omega, v+h^c(v, \omega) \mid v \in H^c
\big\} \\
&=\big\{u_0= e^{z(\omega)}(v+h^c(v, \omega))\mid v \in
H^c
\big\}\\
&=\big\{u_0= v+e^{z(\omega)}h^c(e^{-z(\omega)}v, \omega)\mid  v \in
H^c \big\},
\end{align*}
which implies that $\widetilde {M}^c(\omega)$ is a center
manifold.
\end{proof}
Note that a local center manifolds is not unique~\cite{Carr}.



\section{Smoothness of center  manifolds} \label{Sec4}
In this section, we prove that for each $\omega\in \Omega$\,, $M^c(\omega)$~is
a $C^k$~center manifold.

\begin{theorem} \label{Thm(4.1)}   Assume that $G$~is~$C^k$ in~$u$.
If $\gamma < k\eta < \min\{\beta, \alpha\}$ and
\[
K \Lip_u
G\left(\frac{1}{i\eta-\gamma}+\frac{1}{\beta-i\eta}+\frac{1}{\alpha-i\eta}\right)< 1 \quad\text{for
all }  1\leq i\leq k,
\]
then $M^c(\omega)$~is a $C^k$ center manifold for the
random
evolutionary equation~(\ref{eq(2.5)}),  and $Dh^c(0,\omega)=0$\,.\end{theorem}

\begin{proof} We prove this theorem by induction.
First, we consider $k=1$\,.
Since
\[
K \Lip_u
G\left(\frac{1}{\eta-\gamma}+\frac{1}{\beta-\eta}+\frac{1}{\alpha-\eta}\right)< 1
\]
there exists a small number $\delta>0$ such that $\gamma <
\eta-\eta'<\min\{\beta,\alpha\}$ and for all  $ 0 \leq \eta' \leq 2\delta$\,,
\[
K \Lip_u
G\left[\frac{1}{(\eta-\eta')-\gamma}+\frac{1}{\beta-(\eta-\eta')}+\frac{1}{\alpha-(\eta-\eta')}\right]<
1\,.
\]
Thus, $J^s(\cdot,v)$ defined in the proof of Theorem~\ref{Thm(3.1)} is a uniform contraction in $C_{\eta-\eta'}
\subset  C_\eta$ for any $0\leq \eta'\leq 2\delta$\,.
Therefore,
$u(\cdot, v,\omega)\in C_{\eta-\eta'}$\,.
For $v_0\in H^c$, we
define two operators: let
\[
Sv_0=\Psi_A(t,0)v_0 \,,
\]
and for $v\in C_{\eta-\delta}$ let
\begin{align*}
T v&=\int_0^t \Psi_A(t,s)
P^cD_uG(\theta_s\omega,u(s,v_0,\omega)) v\, ds \\
&\quad {}+\int_{-\infty}^{+\infty} \Psi_0(t,s)
B(t-s)D_uG(\theta_s\omega,u(s,v_0,\omega)) v \,ds\,.
\end{align*}
From the assumption,
$S$~is a bounded linear operator from center subspace~$H^c$ to slowly varying functions space~$C_{\eta-\delta}$.
Using the same arguments in the proof Theorem~\ref{Thm(3.1)}  that
$J^c$~is a contraction, we have that $T$~is a bounded linear
operator from~$C_{\eta-\delta}$ to itself and
\[
\|T\| \leq K \Lip_u
G\left(\frac{1}{\eta-\delta-\gamma}+\frac{1}{\beta-(\eta-\delta)}+\frac{1}{\alpha-(\eta-\delta)}\right)<
1\,,
\]
which implies that the operator $\operatorname{id}-T$ is invertible in~$C_{\eta-\delta}$.
For $v, v_0\in H^c$, we set
\begin{align*}
I={}&\int_0^t \Psi_A(t,s)
P^c\Big[G(\theta_s\omega,u(s,v,\omega))-
G(\theta_s\omega,u(s,v_0,\omega))\\
&\qquad{}-D_uG(\theta_s\omega,u(s,v_0,\omega))
(u(s,v,\omega)-u(s,v_0,\omega))\Big]\,ds\\
&{} +\int^{+\infty}_{-\infty} \Psi_0(t,s)
B(t-s)\Big[G(\theta_s\omega,u(s,v,\omega))-
G(\theta_s\omega,u(s,v_0,\omega))\\
&\qquad{}-D_uG(\theta_s\omega,u(s,v_0,\omega))
(u(s,v,\omega)-u(s,v_0,\omega))\Big]\,ds\,.
\end{align*}

We obtain
\begin{align}\label{eq(4.1)}
\begin{split}
&u(\cdot,v, \omega)-u(\cdot,v_0,
\omega)-T (u(\cdot,v,
\omega)-u(\cdot,v_0,
\omega))
=S(v-v_0) + I,\\
\end{split}
\end{align}
which yields
\begin{align*}
u(\cdot,v, \omega)-u(\cdot,v_0,
\omega)=(\operatorname{id}-T)^{-1}S(v-v_0)+(\operatorname{id}-T)^{-1}I.
\end{align*}
If  $|I|_{C_{\eta-\delta}}=o(|v-v_0|)$ as $v\to
v_0$, then  $u(\cdot,v, \omega) $ is differentiable in~$v$ and its
derivative satisfies $D_v u(t, v,\omega)\in L(H^c,
C_{\eta-\delta})$,  where $L(H^c, C_{\eta-\delta})$ is the
usual space of bounded linear operators and
\begin{align}\label{eq(4.2)}
\begin{split}
D_v u(t, v,\omega)& =\Psi_A(t,0)
v  +\int_0^t \Psi_A(t,s)
P^cD_uG(\theta_s\omega,u(s,v,\omega))D_v u(s,v,\omega)\,ds\\& \quad+\int^{+\infty}_{-\infty} \Psi_0(t,s)
B(t-s)D_uG(\theta_s\omega,u(s,v,\omega)) D_v u(s,v,\omega) \,ds\,.
\end{split}
\end{align}

The tangency condition $Dh^c(0,\omega)=0$ is from  equation~\eqref{eq(4.2)}.

Now we prove that \begin{eqnarray}|I|_{C_{\eta-\delta}} = o(|v-v_0|)\label{Claim1}\end{eqnarray} as
$v\to v_0$\,.
We divide~$I$ into several sufficient small parts.
Let $N$ be a large  positive number to be chosen
later and define the following ten integrals
\begin{align*}
I_1={}&\exp\left[{-(\eta-\delta)|t|-\int_0^t z(\theta_s\omega) \,ds}\right]
\bigg |\int_N^t \Psi_A(t,s)
 P^c\Big[G(\theta_s\omega,u(s,v,\omega))
 \\ &{}
-G(\theta_s\omega,u(s,v_0,\omega))-D_uG(\theta_s\omega,u(s,v_0,\omega))
(u(s,v,\omega) -u(s,v_0,\omega))\Big]\,ds\bigg|
\end{align*}
for $t\geq N$\,.
\begin{align*}
I_1'={}&\exp\left[{-(\eta-\delta)|t|-\int_0^t z(\theta_s\omega) \,ds}\right]
\bigg|\int^{-N}_t \Psi_A(t,s)
P^c\Big[G(\theta_s\omega,u(s,v,\omega))\\&{} -
G(\theta_s\omega,u(s,v_0,\omega))-D_uG(\theta_s\omega,u(s,v_0,\omega))
(u(s,v,\omega)-u(s,v_0,\omega))\Big]\,ds\bigg|
\end{align*}
for $t\leq -N$\,.
\begin{align*}
I_2={}&\exp\left[{-(\eta-\delta)|t|-\int^t_0 z(\theta_s\omega) \,ds}\right]
\bigg|\int_0^N \Psi_A(t,s)
P^c\Big[G(\theta_s\omega,u(s,v,\omega)) \\&{}-
G(\theta_s\omega,u(s,v_0,\omega))-D_uG(\theta_s\omega,u(s,v_0,\omega))
(u(s,v,\omega)-u(s,v_0,\omega))\Big]\,ds\bigg|
\end{align*}
for $0\leq t \leq N$\,.
\begin{align*}
I_2'={}&\exp\left[{-(\eta-\delta)|t|-\int^t_0 z(\theta_s\omega) \,ds}\right]
\bigg|\int_{-N}^0 \Psi_A(t,s)
P^c\Big[G(\theta_s\omega,u(s,v,\omega)) \\&{}-
G(\theta_s\omega,u(s,v_0,\omega))-D_uG(\theta_s\omega,u(s,v_0,\omega))
(u(s,v,\omega)-u(s,v_0,\omega))\Big]\,ds\bigg|
\end{align*}
for $-N \leq t \leq 0$\,.

Let $\overline N$ be a large  positive number to be chosen later.
For
$| t| \leq \overline N$\,, we set
\begin{align*}
I_3={}&\exp\left[{-(\eta-\delta)|t|-\int^t_0 z(\theta_s\omega) \,ds}\right]
\bigg|\int^{\overline N}_t [-\Psi_A(t,s)]
P^u\Big[G(\theta_s\omega,u(s,v,\omega))
\\&{}
-G(\theta_s\omega,u(s,v_0,\omega))-D_uG(\theta_s\omega,u(s,v_0,\omega))
(u(s,v,\omega)-u(s,v_0,\omega))\Big]\,ds\bigg|.
\\
I_3'={}&\exp\left[{-(\eta-\delta)|t|-\int^t_0 z(\theta_s\omega) \,ds}\right]
\bigg|\int_{-\overline N}^t \Psi_A(t,s)
P^s\Big[G(\theta_s\omega,u(s,v,\omega))\\&{} -
G(\theta_s\omega,u(s,v_0,\omega))-D_uG(\theta_s\omega,u(s,v_0,\omega))
(u(s,v,\omega)-u(s,v_0,\omega))\Big]\,ds\bigg|.
\\
I_4={}&\exp\left[{-(\eta-\delta)|t|-\int^t_0 z(\theta_s\omega) \,ds}\right]
\bigg|\int^\infty_{\overline N} [\Psi_A(t,s)]
P^u\Big[G(\theta_s\omega,u(s,v,\omega))\\&{} -
G(\theta_s\omega,u(s,v_0,\omega))-D_uG(\theta_s\omega,u(s,v_0,\omega))
(u(s,v,\omega) -u(s,v_0,\omega))\Big]\,ds\bigg|.
\\
I_4'={}&\exp\left[{-(\eta-\delta)|t|-\int^t_0 z(\theta_s\omega) \,ds}\right]
\bigg|\int_{-\infty}^{-\overline N} \Psi_A(t,s)
P^s\Big[G(\theta_s\omega,u(s,v,\omega))\\&{} -
G(\theta_s\omega,u(s,v_0,\omega))-D_uG(\theta_s\omega,u(s,v_0,\omega))
(u(s,v,\omega) -u(s,v_0,\omega))\Big]\,ds\bigg|.
\end{align*}

For $|t| \geq \overline N$\,, we set
\begin{align*}
I_5={}&\exp\left[{-(\eta-\delta)|t|-\int^t_0 z(\theta_s\omega) \,ds}\right]
\bigg|\int^{+\infty}_{t} [\Psi_A(t,s)]
P^u\Big[G(\theta_s\omega,u(s,v,\omega))\\&{} -
G(\theta_s\omega,u(s,v_0,\omega))-D_uG(\theta_s\omega,u(s,v_0,\omega))
(u(s,v,\omega) -u(s,v_0,\omega))\Big]\,ds\bigg|.
\\
I_5'={}&\exp\left[{-(\eta-\delta)|t|-\int^t_0 z(\theta_s\omega) \,ds}\right]
\bigg|\int_{-\infty}^{t} \Psi_A(t,s)
P^s\Big[G(\theta_s\omega,u(s,v,\omega))\\ &{} -
G(\theta_s\omega,u(s,v_0,\omega))-D_uG(\theta_s\omega,u(s,v_0,\omega))
(u(s,v,\omega)-u(s,v_0,\omega))\Big]\,ds\bigg|.
\end{align*}

It is sufficient to show that for any $\epsilon >0$ there is a
$\sigma
>0$ such that if $ |v-v_0| \leq \sigma $\,, then
$|I|_{C_{\eta-\delta}} \leq \epsilon |v-v_0|$.
Note that
\begin{align*}
|I|_{C_{\eta-\delta}}& \leq \sup_{t\geq N} I_1 + \sup_{N\geq t\geq 0}
I_2+ \sup_{t\leq  -N} I_1' + \sup_{-N\leq t\leq 0}
I_2'+ \sup_{ |t|\leq \overline N} I_3+\sup_{ |t|\leq \overline N} I_4
+ \sup_{|t|\geq \overline N} I_5\\&\quad  +\sup_{ |t|\leq \overline N} I_3'+\sup_{  |t|\leq \overline N} I_4'
+ \sup_{|t |\geq \overline N} I_5'.
\end{align*}
A computation similar to~(\ref{eq(3.7)}) implies that
\begin{align*}
I_1 &\leq 2K\Lip_uG \\& \int_N^t
\exp[{(\gamma-(\eta-\delta))|t-s|}]\exp({-\delta |s|})
|u(\cdot, v,\omega)-u(\cdot,v_0,\omega)|_{C_{\eta-2\delta}} \,ds \\
& \leq \frac{2K^2\Lip_uG \exp({-\delta
N})}{(\eta-\gamma-\delta)\left\{1-K \Lip_u
G\left[\frac{1}{(\eta-2\delta)-\gamma}+\frac{1}{\beta-(\eta-2\delta)}+\frac{1}{\alpha-(\eta-2\delta)}\right]\right\}}
|v - v_0 |.
\end{align*}

 Choose~$N$ so large that
\begin{align*}
 \frac{2K^2\Lip_uG \exp({-\delta
N})}{(\eta-\gamma-\delta)\left\{1-K \Lip_u
G\left[\frac{1}{(\eta-2\delta)-\gamma}+\frac{1}{\beta-(\eta-2\delta)}+\frac{1}{\alpha-(\eta-2\delta)}\right]\right\}}
\leq \frac{1}{8} \epsilon.
\end{align*}

Hence for such~$N$ we have that
\begin{equation*}
\sup_{t\geq N}I_1 \leq \frac{1}{8}\epsilon |v-v_0|.
\end{equation*}
Fixing such~$N$, for~$I_2$ we have that
\begin{align*}
I_2 &\leq K\int_0^N \exp[{(\gamma+(\eta-\delta))|t-s|}]\Big\{
\int^1_0\big |
D_uG(\theta_s\omega,\tau u(s,v,\omega)+(1-\tau)\\
&\quad u(s,v_0,\omega))  -D_uG(\theta_s\omega,u(s,v_0,\omega))\big |\,d\tau\Big\}
|u(\cdot, v,\omega)-u(\cdot,v_0,\omega)|_{C_{\eta-\delta}} \,ds \\
& \leq \frac{K^2|v-v_0|}{1-K \Lip_u
G\left[\frac{1}{\eta-\delta-\gamma}+\frac{1}{\beta-(\eta-\delta)}+\frac{1}{\alpha-(\eta-\delta)}\right]}\\&
\quad \int_0^N \exp[{(\gamma+(\eta-\delta))|t-s|}]\Big\{ \int^1_0\big |
D_uG(\theta_s\omega,\tau
u(s,v,\omega)+(1-\tau)\\&
\quad u(s,v_0,\omega))-D_uG(\theta_s\omega,u(s,v_0,\omega))\big | \,d\tau\Big\}\,ds\,.
\end{align*}
From the continuity of  the integrand in~$(s, v)$,   the last integral is continuous at the point~$v_0$.
Thus, we have that there
is a $\sigma_1
>0$ such that if $|v-v_0| \leq \sigma_1$\,, then
\begin{equation*}
\sup_{N\geq t\geq 0}I_2 \leq \frac{1}{8}\epsilon |v-v_0|.
\end{equation*}
Therefore, if  $|v-v_0| \leq \sigma_1$\,, then
\begin{equation*}
\sup_{t\geq N} I_1 + \sup_{N\geq t\geq 0}
I_2 \leq \frac{1}{4}\epsilon
|v-v_0|.
\end{equation*}
In the same way,   there
is a $\sigma_1'
>0$ such that if  if  $|v-v_0| \leq \sigma_1'$, then
\begin{equation*}
\sup_{t\leq -N} I_1 + \sup_{-N\leq t\leq 0}
I_2 \leq \frac{1}{4}\epsilon
|v-v_0|.
\end{equation*}
Similarly, by choosing~$\overline N$ to be sufficiently large,
\begin{equation*}
\sup_{| t|\leq \overline N} I_4 + \sup_{|t|\geq \overline N}I_5 \leq
\frac{1}{8}\epsilon|v-v_0|,
\end{equation*}
\begin{equation*}
\sup_{| t|\leq\overline N} I_4' + \sup_{|t|\geq \overline N}I_5' \leq
\frac{1}{8}\epsilon|v-v_0|,
\end{equation*}
and for fixed such~$\overline N$, there exists $\sigma_2>0$ such that
if $|v-v_0| \leq \sigma_2$\,, then
\begin{equation*}
\sup_{  |t|\leq \overline N}I_3  \leq \frac{1}{8}\epsilon
|v_1-v_2|
\quad\text{and}\quad
\sup_{| t|\geq \overline N}I_3' \leq \frac{1}{8}\epsilon
|v_1-v_2|.
\end{equation*}
Taking $\sigma = \min \{\sigma_1, \sigma_1', \sigma_2\}$, we have that if
$|v-v_0| \leq \sigma$\,, then
\begin{equation*}
|I|_{C_{\eta-\delta}} \leq  \epsilon |v-v_0|.
\end{equation*}
Therefore  $|I|_{C_{\eta-\delta}} = o(|v-v_0|)$ as $v\to
v_0$\,.

We now prove that $D_v u(t, v,\omega)$~is
continuous from the center subspace~$H^c$ to slowly varying functions space~$C_\eta$.
For $v, v_0 \in H^c$,
using~(\ref{eq(4.2)}),
\begin{align}\label{eq(4.3)}
\begin{split}
&D_v u(t, v,\omega)-D_v u(t,v_0,\omega)\\
&=\int_0^t \Psi_A(t,s)
P^c\Big(D_uG(\theta_s\omega,u(s,v,\omega))D_v
u(s,v,\omega)\\&\qquad{}  -
D_uG(\theta_s\omega,u(s,v_0,\omega))D_v u(s,v_0,\omega)\Big)\,ds\\
&\quad{}+\int_{-\infty}^\infty \Psi_0(t,s)
 B(t-s)
\Big(D_uG(\theta_s\omega,u(s,v,\omega))D_v
u(s,v,\omega)\\&\qquad{}  -
D_uG(\theta_s\omega,u(s,v_0,\omega))D_v
u(s,v_0,\omega)\Big)\,ds \\
&=\int_0^t\Psi_A(t,s)
P^c\Big(D_uG(\theta_s\omega,u(s,v,\omega))
\\&\qquad{} (D_v
u(s,v,\omega)- D_v u(s,v_0,\omega))\Big)\,ds\\
&\quad+\int_{-\infty}^\infty \Psi_0(t,s)
 B(t-s)
\Big(D_uG(\theta_s\omega,u(s,v,\omega))
\\&\qquad{}  (D_v u(s,v,\omega)-
D_v u(s,v_0,\omega))\Big)\,ds + \bar I,
\end{split}
\end{align}
where
\begin{align*} \bar I ={}&
\int_0^t \Psi_A(t,s)
P^c\big(D_uG(\theta_s\omega,u(s,v,\omega))\\ & \qquad{}
-D_uG(\theta_s\omega,u(s,v_0,\omega)
\big)D_v u(s,v_0,\omega)\,ds\\
&{} +\int^{+\infty}_{-\infty} \Psi_0(t,s)
B(t-s)\big(D_uG(\theta_s\omega,u(s,v,\omega))\\ & \qquad{}
-D_uG(\theta_s\omega,u(s,v_0,\omega)
\big)D_v u(s,v_0,\omega)\,ds\,.
\end{align*}
 Then from estimating $|D_v u(\cdot, v,\omega)-D_v u(\cdot,v_0,\omega)|_{L(H^c,
 C_\eta)}$, we have
\begin{align*}
&|D_v u(\cdot, v,\omega)-D_v u(\cdot,v_0,\omega)|_{L(H^c, C_\eta)}  \leq \frac{|\bar I|_{L(H^c, C_\eta)}}{1- K \Lip_u
G\left(\frac{1}{\eta-\gamma}+\frac{1}{\beta-\eta}+\frac{1}{\alpha-\eta}\right)}.
\end{align*}

Using the same argument we used for the last claim of equation~\eqref{Claim1}, we obtain that
$|\bar I|_{L(H^c, C_\eta)} = o(|v-v_0|)$ as $v \to v_0$\,.
Hence
$D_v u(t, v,\omega)$~is continuous from the center space~$H^c$ to bounded linear operators space~${L(H^c,
C_\eta)}$.
Therefore, $u(t,v,\omega)$ is~$C^1$ from~$H^c$ to~$C_\eta$.

Now we show that~$u$ is~$C^k$ from the center space~$H^c$ to slowly varying functions space~$C_{k\eta}$ by induction for $k\ge 2$\,.
By the induction
assumption, we know that~$u$ is~$C^{k-1}$ from~$H^c$ to~$C_{(k-1)\eta}$ and the   $(k-1)$st~derivative~$D^{k-1}_v
u(t, v,\omega)$ satisfies the following equation.
\begin{align*}
D_v^{k-1} u ={}& \int^t_0\Psi_A(t,s)
P^c(D_uG(\theta_s\omega,u)D_v^{k-1} u\, ds\\
&{} + \int^{+\infty}_{-\infty} \Psi_0(t,s)
B(t-s)D_uG(\theta_s\omega,u)D_v^{k-1} u\, ds\\
={}& \int^t_0\Psi_A(t,s)
P^cR_{k-1}(s, v, \omega) \,ds\\
&{} + \int^{+\infty}_{-\infty} \Psi_0(t,s)
B(t-s) R_{k-1}(s, v, \omega) \,ds\,,
\end{align*}
where
\begin{equation*}
R_{k-1}(s,v,\omega) =  \sum_{i=0}^{k-3}\begin{pmatrix} k-2\\i
\end{pmatrix} D_v^{k-2-i} \big(D_uG(\theta_s
\omega,u(s,v,\omega))\big) D_v^{i+1}u(s,v,\omega).
\end{equation*}
Note that $D^i_v u \in C_{i\eta} $ for $i=1,\ldots, k-1$\,,
from the induction hypothesis.
Thus, using~$G$ is~$C^k$, we  verify that $R_{k-1}(\cdot,v,\omega) \in
L^{k-1}\big(H^c,C_{(k-1)\eta}\big)$ and is~$C^1$  with respect to~$v$, where
$L^{k-1}\big(H^c,C_{(k-1)\eta}\big)$ is the usual space of
bounded $k-1$ linear forms.
Since 
\[
K \Lip_u
G\left(\frac{1}{i\eta-\gamma}+\frac{1}{\beta-i\eta}+\frac{1}{\alpha-i\eta}\right)< 1 \quad\text{for
all } 1\leq i\leq k.
\]
The same argument used in the case
$k=1$  shows that $ D^{k-1}_v u(\cdot, v,\omega)$ is~$C^1$ from~$H^c$ to~$L^k(H^c,C_{k\eta})$.
This completes the
proof.
\end{proof}

\begin{theorem} \label{Thm(4.2)}
Assume that the nonlinearity~$F(u)$ in equation~\eqref{eq(2.1)} is  $C^k$~smooth.
If $\gamma < k\eta < \min\{\beta,\alpha\}$ and
\[
K  \Lip
F \left(\frac{1}{i\eta-\gamma}+\frac{1}{\beta-i\eta}+\frac{1}{\alpha-i\eta}\right)< 1 \quad\text{for
all }  1\leq i\leq k,
\]
then
$\widetilde M^c(\omega)=T^{-1}(\omega,M^c(\omega))$  is a $C^k$~center
 manifold for the stochastic partial differential  equation~(\ref{eq(2.1)}).
\end{theorem}
\begin{proof}
From the construction of center manifolds, we obtain
\begin{equation*}
\widetilde M^c(\omega)=\big\{v+\tilde h^c(v, \omega) \mid
v\in H^c\big\},
\end{equation*}
where $\tilde h^c(v, \omega)=\exp[z(\omega)]h^c(\exp[{-z(\omega)]}v, \omega)$.
Note that
$h^c(\cdot, \omega)$ is~$C^k$,  then $\tilde h^c(\cdot, \omega)$ is~$C^k$.
We conclude $\widetilde M^c(\omega)$  is a $C^k$~center
 manifold for the stochastic partial differential  equation~(\ref{eq(2.1)}).

\end{proof}

\section{Exponential attraction  principle} \label{Sec5}
This section proves the stability of the random evolution equation~\eqref{eq(2.5)}.
Our main result is Theorem~\ref{th5.3}, which shows that the dynamic behavior of equation~\eqref{eq(2.5)} is exponentially approximated by the solution of on  stochastic center manifold.
We do not assume $H^s=\emptyset$ or $H^u=\emptyset$\,.
Instead, we project the solution of equation~\eqref{eq(2.5)} to the stable  and unstable subspaces.
Lemma~\ref{lemma9} is a week approach solutions version as it only concerns solutions approach~$M^c(\omega)$.
It does not assert solutions.

\begin{lemma} \label{lemma9}
Let $M^c(\omega)$ be a  locally stochastic center manifold for the stochastic evolution equation~\eqref{eq(2.5)}, then there exist positive constants~$U$  and~$\mu$ such that for $\omega \in \Omega$\,,
\begin{eqnarray*}
|P^su(t,u_0, \omega)-P^sh^c(P^cu(t,u_0,\omega),\omega)|\leq U\exp({-\mu t})|P^su_0-P^sh^c(P^cu_0,\omega)|,
\end{eqnarray*}
$t\geq 0$ large enough.
\end{lemma}
\begin{proof}
Let $X(t,u_0,\omega)=P^su(t,u_0, \omega)-P^sh^c(P^cu(t,u_0,\omega),\omega)$,  then $X$ satisfies the equation
\begin{equation}
\frac {d{X(t,u_0,\omega)}}{dt }=AX+z(\theta_t\omega)X+N( X, \theta_t\omega),\end{equation}
where
\begin{align*}
N(X,\theta_t\omega)={}&P^sG(\theta_t\omega,(\operatorname{id}-P^s)u+X+P^sh^c(P^cu(t,u_0,\omega),\omega))\\&{} -P^sG(\theta_t\omega,(\operatorname{id}-P^s)u+P^sh^c(P^cu(t,u_0,\omega),\omega))\\
&{} +DP^sh^c(P^cu(t,u_0,\omega)),\omega)[
P^cG(\theta_t\omega,(\operatorname{id}-P^s)u\\ \quad &{} +P^sh^c(P^cu(t,u_0,\omega),\omega))\\ \quad&{} -P^cG(\theta_t\omega,(\operatorname{id}-P^s)u+X+P^sh^c(P^cu(t,u_0,\omega),\omega))].
\end{align*}
Since the derivative of  stochastic center manifold is  less than Lipschitz constant~$ \Lip_uG$ and $h^c$~is Lipschitz continuous,  we have  $|N(\omega, X)|<U_1 |X|$ for a positive constant $U_1$.
We obtain
\begin{eqnarray*}
X(t,u_0,\omega)&=&\Psi_A(t,0)X(0,u_0,\omega)+\int_0^t\Psi_A(t,s)N(X,\omega)\,ds\,.\end{eqnarray*}
The conclusion of Lemma~\ref{lem(2.1)} yields  $\int_0^tz(\theta_r\omega)dr<\epsilon t$\,,   $\int_s^tz(\theta_r\omega)dr<\epsilon(t-s)$  for $t$~large enough and  $\beta>\epsilon>0$\,.
Then
\begin{eqnarray*}
|X(t,u_0,\omega)&|\leq &K \exp[{(-\beta+\epsilon) t}]|X(0,u_0,\omega)|\\
&&{}+U_1K\int_0^t\exp[{(-\beta+\epsilon) (t-s)}|X(s,u_0,\omega)|\,ds\,.
\end{eqnarray*}
From  Gronwall's inequality we have the lemma.
\end{proof}
\begin{remark}Similarly we have the conclusion there exist  positive constants~$V$ and~$\nu$ such that  for $\omega \in \Omega$,\begin{equation*}
|P^uu(t,u_0, \omega)-P^uh^c(P^cu(t,u_0,\omega),\omega)|\leq V\exp({\nu t})|P^uu_0-P^uh^c(v,\omega)|,
\end{equation*}
 $t\leq 0$ small enough.
In the finite dimensional case~\cite[Theorem 7.1]{Box}, the conclusions are expressed as the random norm.
\end{remark}
Now we consider the asymptotic behavior of the  stochastic evolution equation~\eqref{eq(2.5)}.
The stochastic dynamical system on the stochastic center manifold  is governed by the following random evolution equation.
\begin{equation}\label{th2}
\dot{v}=P^cAv+z(\theta_t\omega)v+P^cG(\theta_t\omega, v+h^c(v,\omega)).
\end{equation}
If the subspace~$H^c$ is finite dimension,   then equation~\eqref{th2} is an ordinary differential equation.
The following theorem shows that the solution of  stochastic evolution equation~\eqref{eq(2.5)} is exponentially approximated by the solution of equation~\eqref{th2}.
\begin{theorem}\label{th5.3}
Let~$M^c(\omega)$ be a  locally stochastic center manifold for stochastic evolution equation~\eqref{eq(2.5)}, then there are positive  random variables $L_1(\omega)$~and~$L_2(\omega)$ and constants $l_1$~and~$l_2$ such that for initial values $u_0\in H$ and $v_0\in M^c$, the following two conditions hold:
\begin{eqnarray}
&&|P^cu(t,u_0,\omega)-v(t,v_0,\omega)|\leq L_1(\omega) \exp( {-l_1 t}),\label{th31} \\
&&|P^su(t,u_0,\omega)-P^s h^c(v(t,v_0,\omega),\omega)|\leq L_2(\omega) \exp({ -l_2 t}),\label{th311}
\end{eqnarray}
 $t\geq 0$ large enough and $\omega\in \Omega$.
\end{theorem}
\begin{proof} We construct a differential equation about the error of the two solutions on centre subspace.
Let $X'(t,u_0,\omega)=P^cu(t,u_0, \omega)-v(t,v_0,\omega)$, which is the errors of the two solutions on centre subspace,  then $X'$ satisfies the equation
\begin{equation}\label{eq31}
\frac {d{X'(t,u_0,\omega)}}{dt }=AX'+z(\theta_t\omega)X'+N'(X',\omega),\end{equation}
where the nonlinear term
\begin{align*}
N'(X',\omega)={}&P^cG(\theta_t\omega,u(t,u_0,\omega))-P^cG(\theta_t\omega, v+h^c(v,\omega)).
\end{align*}
Rewrite the nonlinearity $N'$ as three steps, so that $N'$ is  bounded.
\begin{align*}
N'(X',\omega)={}&P^cG(\theta_t\omega,u(t,u_0,\omega))-P^cG(\theta_t\omega, P^cu+h^c(P^cu,\omega))\\
&{} +P^cG(\theta_t\omega, P^cu+h^c(P^cu,\omega))-P^cG(\theta_t\omega, P^cu+h^c(v,\omega))\\
&{} +P^cG(\theta_t\omega, P^cu+h^c(v,\omega))-P^cG(\theta_t\omega, v+h^c(v,\omega)).
\end{align*}
Rewrite equation~\eqref{eq31} as an integral equation.
\begin{align*}
X'(t,u_0,\omega)=\Psi_A(t,0)X'(0,u_0,\omega)+\int_0^t\Psi_A(t,s)N'(X',\omega)\,ds\,.\end{align*}
Since  $G$~is Lipschitz continuous,  there exist positive constants $V_1$,  $v_1$ and ~$V_2$ such that
\begin{align*}
|N'(X',\omega)|\leq &V_1\exp(-v_1 t)(|P^su_0-P^sh^c(P^cu_0,\omega)|+|P^uu_0-P^uh^c(P^cu_0,\omega)|)\\ &{}+V_2|X'|.
\end{align*}
Lemma~\ref{lem(2.1)} yields  $\int_0^tz(\theta_r\omega)dr<\epsilon t$ and  $\int_s^tz(\theta_r\omega)dr<\epsilon(t-s)$  for $t$~large enough and  $\gamma>\epsilon>0$\,.
It follows that
\begin{align*}
|X'(t,u_0,\omega)|\leq& K \exp[{(\gamma+\epsilon) t}]|X'(0,u_0,\omega)|\\\quad \quad&{} -K(\gamma+\epsilon+v_1)^{-1} \exp({-v_1 t})V_1(|P^su_0-P^sh^c(P^cu_0,\omega)|\\\quad\quad&{}+|P^uu_0-P^uh^c(P^cu_0,\omega)|)\\&{} +V_2K\int_0^t\exp[{(\gamma+\epsilon) (t-s)}|X'(s,u_0,\omega)|\,ds\,.\end{align*}
From  Gronwall's inequality we have inequality~\eqref{th31}.

Now we prove inequality~\eqref{th311},
\begin{align*}
&|P^su(t,u_0,\omega)-P^s h^c(v(t,v_0,\omega),\omega)|\\
&\leq |P^su(t,u_0,\omega)-P^s h^c(P^cu,\omega)|+|P^s h^c(P^cu,\omega)-P^s h^c(v(t,v_0,\omega),\omega)|.
\end{align*}
By   Lemma~\ref{lemma9} and inequality~\eqref{th31} we have inequality~\eqref{th311}.
\end{proof}
\begin{remark}\label{rem1}Similarly we conclude there exist  positive random variable~$L_3(\omega)$ and positive constant~$l_3$ such that \begin{eqnarray}
&&|P^cu(t,u_0,\omega)-v(t,v_0,\omega)|\leq L_3(\omega) \exp( {l_3 t}), \\
&&|P^uu(t,u_0,\omega)-P^u h^c(v(t,v_0,\omega),\omega)|\leq L_4(\omega) \exp({ l_4 t}),
\end{eqnarray}
 $t\leq 0$ small enough and $\omega\in \Omega$.
\end{remark}
\begin{corollary} \label{cor5.4}
Let $H^u=\emptyset$.
Suppose that the zero solution of equation~\eqref{th2} is  asymptotically stable and the initial value $u_0$ is sufficiently small.
Then the zero solution of  equation~\eqref{eq(2.5)} is asymptotically  stable. \end{corollary}
\begin{proof}
From  Theorem~\ref{th5.3},  we conclude that the zero   solution of equation~\eqref{th2} is asymptotically stable on centre and stable  subspaces respectively.
This establishes Corollary~\ref{cor5.4}.
\end{proof}

\section{Approximation of the center manifolds} \label{Sec6}
In this section, we prove the existence and  approximation of stochastic center manifolds. We avoid  assuming the noise is differentiable in the proof process~\cite{Box1}. Denote~$h(v(t,v_0,\omega),\omega)=:h(v,\omega)$ for the initial value $P^cu(t,u_0,\omega)=:v(t,v_0,\omega)$.
For a $\mathcal{F}$-measurable function $g(v,\omega):H^c\to  H^u\oplus H^s$ which are~$C^1$ in a neighbourhood of origin  in~$H^c$,
define
\begin{align*}
M(g)(v,\omega)={}&g_t(v,\omega)+g_v(v,\omega)[P^cAv+z(\theta_t\omega)v+P^cG(\theta_t\omega,v+ g)]
\\&{}
-Ag-z(\theta_t\omega)g-P^s G(\theta_t\omega, v+g)-P^uG(\theta_t\omega, v+ g),
\end{align*}
where $v\in H^c$.
\begin{theorem}\label{k12}
Suppose that \(g\)~is a  $\mathcal{F}$-measurable function such that $g(0,\omega)=0$\,, $g_v(0,\omega)$ $=0$ and $M(g)(v,\omega)=\Ord{|v|^q}$ for some $q>1$\,.
Then as $v\to 0$\,,
\begin{equation*}|h^c(v,\omega)-g(v,\omega)|=\Ord{|v|^q}.\end{equation*}
\end{theorem}
\begin{proof}
Let an  $\mathcal{F}$-measurable function  $y(s,v,\omega)\in C_\eta(\omega)$ and let $y(s,\cdot,\omega):H^c\to H^u\oplus H^s$ be continuous differential with compact support such that
$y(0,0,\omega)=g(0,\omega)=0$\,,  $y_v(0,0,\omega)=g_v(0,\omega)=0$ and $y_v(0,v,\omega)=g_v(v,\omega)$ for $|v|$~small enough.
The function~$y(s,v,\omega)$ connects the functions~$h^c(v,\omega)$ and~$g(v,\omega)$.
Set
\begin{align*}
M(y)(0,v,\omega)={}&y_t(0,v,\omega) +y_v(0,v,\omega)[P^cAv+z(\theta_t\omega)v\\&{}+P^cG(\theta_t\omega, y(0,v,\omega))] -Ay(0,v,\omega)-z(\theta_t\omega)y(0,v,\omega)\\&{}-P^s G(\theta_t\omega, y(0,v,\omega)) -P^uG(\theta_t\omega, y(0,v,\omega)),\end{align*}
where $v\in H^c$.
For a  $\mathcal{F}$-measurable function~$w(s,v,\omega)$, define the operator
\begin{equation*}Sw=J^c (w(s,v,\omega)+y(s,v,\omega),v)-y(s,v,\omega),\end{equation*}
where the operator~$J^c$ is defined by
\begin{align}\label{eq(3.221)}\begin{split}
J^c(u,v):={}&\exp(As)v+\int_0^s
\exp[A(s-r)]
P^cG(\theta_{r}\omega,u(r))\,dr\\&{}
+\int^{+\infty}_{-\infty}B(s-r)\exp(\int_r^sz(\theta_{l}\omega)\,dl)G(\theta_{r}\omega,u(r))\,dr\,.
\end{split}
\end{align}
For a random variable $K(\omega)$, define the domain~$Y$
\begin{equation*}Y(\omega)=\left\{w(s,v,\omega)\in C_\eta(\omega)\mid|w(0,v,\omega)|\leq K(\omega) |v|^q\quad \mbox{for all } v\in H^c\right\}.\end{equation*}
From its construction,  $Y(\omega)$~is closed in~$C_\eta(\omega)$.
Since $J^c$~is a contraction mapping on~$C_\eta(\omega)$, $S$~is a contracting mapping on~$Y(\omega)$.
We prove that there exists a random variable~$K(\omega)$ such that $S$ maps~$Y(\omega)$ into~$Y(\omega)$.
Then by the uniqueness of fixed points and definition of~$h^c$ we conclude the theorem.

 Now we prove $Y(\omega)$ exists.
For $w\in C_\eta(\omega)$,  let $u_c(s,v,\omega)$ be the solution of
\begin{align}\label{eq(6.1)}
&  \frac {du_c} {dt}=Au_c+z(\theta_{s}\omega)u_c+G(\theta_{s}\omega,y(0,u_c,\omega)+w(0,u_c,\omega)),\nonumber\\ &u_c(0,v,\omega)=v.
 \end{align}
     From equation~\eqref{eq(3.221)} and the construction of $h(v,\omega)$,
\begin{align*}
&J^c(y(0,v,\omega)+w(0,v,\omega),0)=\int^{+\infty}_{-\infty}B(-r)\exp\left[\int_r^0 z(\theta_{l}\omega)\,dl\right]\\& G(\theta_{r}\omega,y(0,u_c,\omega)+w(0,u_c,\omega))\,dr\,.
\end{align*}
Since the slowly varying function $y(r,u_c,\omega)\in C_\eta(\omega)$\,,
\begin{align*}&
\left|\exp\left[-Ar+\int_r^0 z(\theta_{l}\omega)\,dl\right]P^sy(0, u_c,
\omega)\right|
\\ &\leq K\exp[(\beta-\eta)r] |y(r,u_c,\omega)|_{C_\eta(\omega)}
\\&\to 0 \quad\text{as }r \to -\infty.
\end{align*}
Also
\begin{align*}&
\left|\exp\left[-Ar+\int_r^0 z(\theta_{l}\omega)\,dl\right]P^uy(0, u_c,
\omega)\right|
\\ &\leq K\exp[-(\alpha-\eta)r] |y(r,u_c,\omega)|_{C_\eta(\omega)}
\\&\to 0 \quad\text{as }r \to +\infty.
\end{align*}
 So
\begin{align*}
&y(0,v,\omega)\\={}&\int_{-\infty}^0\frac d{dr}\left[\exp\left[-Ar+\int_r^0 z(\theta_{l}\omega)\,dl\right]P^sy(0,u_c,\omega)\right]\,dr
\\&{} -\int^{\infty}_0\frac d{dr}\left[\exp\left[-Ar+\int_r^0 z(\theta_{l}\omega)\,dl\right]P^uy(0,u_c,\omega)\right]\,dr
\\  ={}&\int_{-\infty}^0\ \exp\left[-Ar+\int_r^0 z(\theta_{l}\omega)\,dl\right]\bigg[(-A-z(\theta_r\omega))P^sy(0,u_c,\omega)
\\&{} +\frac d{dr}P^sy(0,u_c,\omega)\bigg]\,dr-\int^{\infty}_0\exp\left[-Ar+\int_r^0 z(\theta_{l}\omega)\,dl\right]\\\ &
\left[(-A-z(\theta_r\omega))P^uy(0,u_c,\omega)
 +\frac d{dr}P^uy(0,u_c,\omega)\right]\,dr\,.
 \end{align*}
Here
\begin{align*}
& (-A-z(\theta_r\omega))P^sy(0,u_c(r,v,\omega),\omega)+\frac d{dr}P^sy(0,u_c(r,v,\omega),\omega)\\
={}&(-A-z(\theta_r\omega))P^sy(0,u_c(r,v,\omega),\omega)+P^sy_v(0,u_c(r,v,\omega),\omega)[Au_c\\&{}+z(\theta_r\omega)u_c+P^cG(\theta_r\omega,y(0,u_c,\omega)+w(0,u_c,\omega))]\\&{}+P^sy_r(0,u_c(r,v,\omega),\omega)\\
={}&M(P^sy)-P^sy_v(0,u_c(r,v,\omega),\omega)[Au_c+z(\theta_r\omega)u_c\\&{}+P^cG(\theta_r\omega,y(0,u_c,\omega))] +P^sG(\theta_r\omega,P^sy(0,u_c,\omega))\\&+P^sy_v(0,u_c(r,v,\omega),\omega) \\&\left[Au_c+z(\theta_r\omega)u_c+P^cG(\theta_r\omega,y(0,u_c,\omega)+w(0,u_c,\omega))\right]\\
={}&M(P^sy)-P^sy_v(0,u_c(r,v,\omega),\omega)[P^cG(\theta_r\omega,y(0,u_c,\omega))\\ &{}-P^cG(\theta_r\omega,y(0,u_c,\omega) +w(0,u_c,\omega))]+P^sG(\theta_r\omega,y(0,u_c,\omega)).   \end{align*}
\begin{align*}  & (-A-z(\theta_r\omega))P^uy(0,u_c(r,v,\omega),\omega)+\frac d{dr}P^uy(0,u_c(r,v,\omega),\omega)\\
={}&(-A-z(\theta_r\omega))P^uy(0,u_c(r,v,\omega),\omega)+P^uy_v(0,u_c(r,v,\omega),\omega)\\&{}[Au_c+z(\theta_r\omega)u_c+P^cG(\theta_r\omega,y(0,u_c,\omega)+w(0,u_c,\omega))]\\&{}+P^uy_r(0,u_c(r,v,\omega),\omega)\\={}&M(P^uy)-P^uy_v(0,u_c(s,v,\omega),\omega)[P^cG(\omega,y(0,u_c,\omega))\\&{} -P^cG(\theta_r\theta_r\omega,y(0,u_c,\omega)+w(0,u_c,\omega))]+P^uG(\theta_r\omega,y(0,u_c,\omega)).  \end{align*}
From the above calculations
\begin{align*}
&J^c(y(0,v,\omega)+w(0,v,\omega),0)-y(0,v,\omega)\\
&=\int_{-\infty}^0\exp\left[-Ar+\int_r^0 z(\theta_{l}\omega)\,dl\right]Q^s\,dr
+\int^{\infty}_0\exp\left[-Ar+\int_r^0 z(\theta_{l}\omega)\,dl\right]Q^u\,dr\,,
 \end{align*}
 where $Q^s$ and $Q^u$ is \begin{align*}
 Q^s(y,w,\omega)={}&P^sy_v(0,u_c(r,v,\omega),\omega)[P^cG(\theta_r\omega,y(0,u_c,\omega))\\&{}-P^cG(\theta_r\omega,y(0,u_c,\omega) +w(0,u_c,\omega))]
-M(P^sy)\\&{} +[P^sG(\theta_r\omega,y(0,u_c,\omega)+w(0,u_c,\omega))-P^sG(\theta_r\omega,y(0,u_c,\omega))],\\
Q^u(y,w,\omega)={}&M(P^uy)-P^uy_v(0,u_c(r,v,\omega),\omega)[P^cG(\theta_r\omega,y(0,u_c,\omega))\\&\-P^cG(\theta_r\omega,y(0,u_c,\omega)+w(0,u_c,\omega))]
-[P^uG(\theta_r\omega,y(0,u_c,\omega)\\&{} +w(0,u_c,\omega))-P^uG(\theta_r\omega,y(0,u_c,\omega))].
\end{align*}
Since $g_v(0,\omega)=0$ and $g$ is ~$C^1$ in a neighborhood of origin  in~$H^c$, for a small random variable~$\epsilon(\omega) $,
\begin{eqnarray*}|g_v(v,\omega)|<\epsilon(\omega)\end{eqnarray*} if $|v| $ small enough.
From $M(y)(0,u_c,\omega)=C_1(\omega)|u_c|^q$ for a positive random variable~$C_1(\omega)$ and all $u_c\in H^c$, we obtain
$|Q^s|\leq C_2(\omega)|u_c|^q$ and  $|Q^u|\leq C_3(\omega)|u_c|^q$ for some random variables $C_2(\omega)$, $C_3(\omega)$,.
Since $u_c$ is the solution of  equation~\eqref{eq(6.1)}, from the Gronwall's  inequality  we  conclude  $|u_c|\leq C_4(\omega)\exp(-cr)|v|$ for a positive constant $c$ and random variable $C_4(\omega)$.

Since the random variables
\begin{eqnarray*}\int_{-\infty}^0\exp(\int_r^0 z(\theta_{l}\omega)\,dl)\,dr<+\infty, \int^{\infty}_0\exp(\int_r^0 z(\theta_{l}\omega)\,dl)\,dr<+\infty, \end{eqnarray*}
so $|Sw(0,v,\omega)|=|J^c(w(0,v,\omega)+y(0,v,\omega),0)-y(0,v,\omega)|\leq C_5(\omega)|u_c|^q\leq K(\omega)|v|^q$ by selecting suitable random variables $C_5(\omega)$, $K(\omega)$.
\end{proof}

\section{Examples show our results} \label{Sec7}
Example~\ref{exam1} shows that we can use computer algebra~\cite{Rob} to directly compute stochastic center manifolds in a wide class of \spde{}s.
Example~\ref{sec:sdwe} establishes that a semilinear damped wave equation has a stochastic center manifold.
Lastly, Example~\ref{sec:idcm} explores a case where the stochastic center manifold itself is infinite dimensional.

\subsection{Stochastic reaction-diffusion modelling}
\label{exam1}

Consider the stochastic reaction-diffusion equation
\begin{equation}\label{pa}
 u_{t}= u_{xx}+u- a  u^3+ \sigma u\circ\dot{W}
\quad\text{with}\quad
 u(0,t)=u(\pi,t)=0\,,  \quad t\geq 0\,.
\end{equation}
We set the space $H:=L^2(0,\pi)$, the operator $Au= {d^2u}/{dx^2}+u$ with domain $D(A):=H^2(0,\pi)\cap H_0^1(0,\pi)$.
Then
the spectrum of the operator~$A$ is
\begin{equation*}\sigma(A)=\{0,-3,-8,-15,\ldots\}=\{1-n^2\mid n=1,2,\ldots\}. \end{equation*}
The corresponding complete and orthogonal  eigenfunctions  are~$\sin nx$\,, $n=1,2,\ldots$\,.
The operator~$A$  satisfies the exponential trichotomy (Definition~\ref{trich}): the center subspace
$H^c=\operatorname{span}\{\sin x\}$, the stable subspace $H^s=\operatorname{span}\{\sin nx,$ $n=2,3,\ldots\}$,  and the unstable subspace $H^u=\emptyset$.

Recall that the random variable $z(\theta_t\omega)$ is the stationary solution of  the Ornstein--Uhlenbeck \sde~\eqref{eq(2.4)}.
In order to transform the random parabolic equation~\eqref{pa}, from the coordinate transform~\eqref{eq(2.7)} we let the transform $u=u^*\exp(\sigma z(\theta_t\omega))$.
For simplicity omit the~$*$ from here on,
the transformed random parabolic equation is
\begin{equation}\label{pa1}
 u_{t}= u_{xx}+u+\sigma z(\theta_t\omega)u-a\exp[2\sigma z(\theta_t\omega)] u^3.
\end{equation}

Select the nonlinearity parameter~$a$ small enough such the conditions of  Theorem~\ref{Thm(3.1)} is satisfied.
Then Theorem~\ref{Thm(3.1)} assures us that there exists a slow (center) manifold
\begin{equation*}M^c(\omega)=\{s\sin x+h^c(s\sin x,\omega)\}=\left \{s \sin x + \sum_{n=2}^\infty c_n(s,\omega)\sin nx\right \},\end{equation*}
where $c_n(s,\omega)=\Ord{s^3}$ as $s\to 0$ for $\omega\in \Omega$\,.
The random dynamical system on the slow manifold~$M^c(\omega)$ is governed by the following equation.

\begin{equation}
\frac {ds}{dt} \sin x=s\sigma z(\theta_t\omega)\sin x
-\frac 3 4{\exp[2\sigma z(\theta_t\omega)]} as^3\sin x
+\Ord{s^5}.  \label{pa2}
\end{equation}
Equation \eqref{pa2} has a common factor of $\sin x$ in each term, it follows that
\begin{equation*}
\frac {ds}{dt}=s\sigma z(\theta_t\omega)-\frac 3 4 {\exp[2\sigma z(\theta_t\omega)]}  as^3+\Ord{s^5}.
\end{equation*}
We now  calculate the slow manifold~$h^c(s\sin x,\omega)$.
For $\omega\in \Omega$, $c_n(s,\omega)$ is odd functions of  $s$ since $-u$ is a solution of equation~\eqref{pa1}. The coefficient $c_n(s,\omega)$ is zero for odd term,
so
\begin{equation*}h^c(s\sin x,\omega)= c_3(s,\omega)\sin 3x+\Ord{s^5}.
\end{equation*}
From $c_3(s,\omega)=\Ord{s^3}$,  we obtain 
\begin{eqnarray*}
a\exp[2\sigma z(\theta_t\omega)](s\sin x+c_3\sin 3x)^3
 = {as^3}
 \exp[2\sigma z(\theta_t\omega)]\sin ^3 x+\Ord{s^5}.\end{eqnarray*}
So \begin{align*}&P^cG(\theta_t\omega,s\sin x+h^c)=\frac 3 4 as^3\exp[2\sigma z(\theta_t\omega)]\sin x+\Ord{s^5},
\\&
P^sG(\theta_t\omega,s\sin x+h^c)=-\frac 1 {4} as^3\exp[2\sigma z(\theta_t\omega)]\sin 3x+\Ord{s^5}.\end{align*}
From the formulation of approximation of  center manifold,
\begin{align*}
M(h^c)(v,\omega)={}&\,h_t^c(v,\omega)+h_v^c(v,\omega)[P^cAv+z(\theta_t\omega)v+P^cG(\theta_t\omega, v+h^c)]-Ah^c\\&{}-z(\theta_t\omega)h^c -P^s G(\theta_t\omega,  v+h^c)-P^uG(\theta_t\omega,  v+h^c),
\end{align*}
there exists an stochastic approximation of center manifold of equation \eqref{pa1}.

Now we directly use  computer algebra to compute the stochastic center manifold~\cite{Rob}.
Define the Ornstein--Uhlenbeck processes $\phi_k(t)=\exp[-(k^2-1)t]\star \dot W(t)$,
 then
\begin{eqnarray}&&
  u= s\sin x + \frac 1 {32} as^3 \sin 3x -\frac 1 {16} as^3\sigma \phi_3(t)\sin 3x+\Ord{s^5, a^2,\sigma^2}\label{cas2}\\
&&\mbox{such that}\quad \frac {ds}{dt}= s\sigma \dot W(t) -\frac 3 4 as^3 +  \Ord{s^5, a^3,\sigma^2}.  \nonumber
 \end{eqnarray}

\subsection{Semilinear damped wave dynamics}
\label{sec:sdwe}

Wang and Duan~\cite{ Wang}  discussed a semilinear wave equation with a two dimensional center-unstable manifold.
We establish that the   semilinear damped wave equation has a one dimensional center manifold.    

Consider the \spde
\begin{equation}\label{HypM}
 u_{tt}+ u_t= \frac 1 4 u_{xx}+u+f(u)+  u\circ\dot{W},
  \quad x\in (0,2\pi),  \quad t\geq 0\,,
\end{equation}
with $u(0,t)=u(2\pi, t)=0$\,.
$W(t)$~is the standard $\mathbb{R}$-valued Wiener process on a probability space~$(\Omega,\mathcal{F}, \mathbb{P} )$.
The nonlinearity term $f$ is smooth and $f(0)=0$.

Let the space ${H}=H_0^1(0,2\pi)\times L^2(0,2\pi)$.
Rewrite the
\spde~(\ref{HypM}) as the following first order stochastic
evolution equations in~${H}$:
\begin{eqnarray}
du&=&v\,dt\,,\label{shyp1}\\
dv&=&\big[ \frac 1 4   u_{xx}+ u- v+f(u)\big]dt+  u\circ dW(t). \label{shyp2}
\end{eqnarray}

First we prove that the system~(\ref{shyp1})--(\ref{shyp2}) generates
a continuous random dynamical system in~${H}$.
Let
$\Psi_1(t)=u(t)$, $\Psi_2(t)=v(t)-  u(t) z(\theta_t\omega)$, where
$z(\omega)$~is the stationary solution of~(\ref{eq(2.4)}) with $\mu=1$.
Then equations \eqref{shyp1} and \eqref{shyp2} are transformed to
 the following random evolution equations
\begin{align}
d\Psi_1&=\big[\Psi_2+  \Psi_1 z(\theta_t\omega)\big]dt\,,\label{rhpy1}\\
d\Psi_2&=\big[\frac 1 4\frac {d^2\Psi_1}{dx^2}+\Psi_1-\Psi_2+f(\Psi_1)\big]dt+\big[z(\theta_t\omega)\Psi_1- z^2(\theta_t\omega)\Psi_1
-  z(\theta_t\omega)\Psi_2
\big]dt\label{rhyp2}\,.
\end{align}
Let $\Psi(t,\omega)=(\Psi_1(t,\omega),
\Psi_2(t,\omega))\in H$\,, then $\varphi(t, \Psi(0, \omega),\omega)=\Psi(t,\omega)$
defines a continuous random dynamical system in~${H}$.
Notice that the stochastic system~(\ref{shyp1})--(\ref{shyp2}) is
conjugated to the random system~(\ref{rhpy1})--(\ref{rhyp2}) by the
homeomorphism
\begin{equation*}
T(\omega, (u, v))=(u, v+uz(\omega)), \quad (u, v)\in{H}
\end{equation*}
with inverse
\begin{equation*}
T^{-1}(\omega, (u, v))=(u, v-uz(\omega)), \quad (u,
v)\in{H}\,.
\end{equation*}
From Lemma \ref{lem(2.2)}, $\hat{\varphi}(t,\omega, (u_0, v_0))=T(\theta_t\omega,
\varphi(t,\omega, T^{-1}(\omega,(u_0, v_0) )))$ is a random
dynamical system generated by~(\ref{shyp1})--(\ref{shyp2}).
Define \begin{align*}&
 \mathcal{A}=\begin{bmatrix}     0, &  {\operatorname{id}}  \\
    \frac 1 4 \frac {\partial^2}{\partial x^2}+\operatorname{id}, &  -\operatorname{id} \end{bmatrix},
    \quad F(\Psi)= \begin{bmatrix} 0   \\
    f(\Psi_1) \end{bmatrix},
    \\&
     Z(\theta_t\omega)=\begin{bmatrix} z(\theta_t\omega), & 0 \\
 -  z^2(\theta_t\omega), & -
z(\theta_t\omega) \end{bmatrix},
\end{align*}
where $\operatorname{id}$~is the identity operator on the Hilbert
space~$L^2(0,2\pi)$.
Then the random system~(\ref{rhpy1})--(\ref{rhyp2}) can be written as
\begin{equation}\label{RHPY}
\frac{d\Psi}{dt}=\mathcal{A}\Psi+Z(\theta_t\omega)\Psi+F(\Psi).
\end{equation}
The operator~$\mathcal{A}$ has the eigenvalues
\begin{equation*}
\delta_k^\pm=-\frac{1}{2}\pm\sqrt{\frac{5}{4}-\frac 1 4
k^2}, \quad k=1,2,\ldots\,,
\end{equation*}
and corresponding eigenvectors are $q^{\pm}_k=(1, \delta_k^\pm)\sin kx $\,.
 The operator~$\mathcal{A}$ has one zero
eigenvalue $\delta_2^+=0$\,, one positive eigenvalue
$\delta_1^+=\frac{1}{2}$ and the others are all complex numbers
with negative real part.
By using a new inner product~\cite{Wang},  we  define an equivalent norm on~${H}$ with the eigenvectors $\{(1, \delta_k^\pm)\sin kx \}$ orthogonal in this norm.

A   calculation yields that, in terms of the new norm, the
semigroup~$\mathcal{S}(t)$ generated by~$\mathcal{A}$
satisfies the exponent trichotomy with $\alpha=\frac 1 2+\epsilon$\,, $\gamma=\frac 1 4 +\epsilon$
$\beta=\frac{1}{2}-\epsilon$ for parameter~$\epsilon$ small enough, and $\dim H^c=\dim {H}^{u}=1$\,. 

Let $\Psi=rq_1^++sq_2^++y$\,, $y\in P^s H$\,.  For application the theory of stochastic center manifolds,
projecting the random system~\eqref{RHPY} to the center subspace span$\{q^+_2\}$ by the projection $P^c$, 
\begin{eqnarray}
\dot{s}q_2^+=Z(\theta_t\omega)sq_2^++P^cF(\theta_t\omega, rq_1^++sq_2^++y).
\end{eqnarray}
By a cut-off function,  $\Lip  F$ is sufficient small.
From Theorem~\ref{Thm(3.1)} we obtain a center manifold~$M(\omega)$ which we denote as the graph of a
random Lipschitz map~$h^c(sq_2^+, \omega)$.

The asymptotic behavior of the random system~\eqref{RHPY} is governed by  the one dimensional equation
\begin{eqnarray} \label{eqe}
\dot{s}q_2^+=Z(\theta_t\omega)sq_2^++P^cF(\theta_t\omega, sq_2^++h^c(sq_2^+, \omega)).
\end{eqnarray}
From Theorem~\ref{th5.3}, 
we know that  the long time  asymptotic behavior of the random system~\eqref{RHPY} is the same as the equation \eqref{eqe}.

\subsection{An infinite dimensional center manifold}
\label{sec:idcm}

There are very few rigorous results involving infinite dimensional center manifolds except the existence theorem of Gallay~\cite{Gal} and the cylindrical analysis of Mielke~\cite{Mie}.
However, applied mathematicians and engineers implicitly, in effect, assume an infinite dimensional center manifold whenever they invoke the method of multiple scales for slowly varying solutions~\cite[e.g.]{Roberts88a, Roberts92c}.
As a basic example let's consider the coupled \spde{}s
\begin{equation}\label{1}
\D tu=au-uv+\sigma u \circ \dot W\,,\quad
\D tv=-v+\DD xv+u^2-2\cK(u^2v)+\sigma v \circ \dot W\,,
\end{equation}
with Neumann boundary conditions $v_z=0$ at $z=0,\pi$\,, and the  operator $\cK=(1+2a-\partial_{zz})^{-1}$.
For example, for the specific case $a=0$,
\begin{equation*}
\cKo v=\int_0^\pi K_0(z,\zeta)v(\zeta)\,d\zeta
\quad\text{where }K_0=\begin{cases}
\frac{\cos(\pi-\zeta)\cos z}{\sin\pi}\,,&z<\zeta\,,\\
\frac{\cos(\pi-z)\cos\zeta}{\sin\pi}\,,&z>\zeta\,,
\end{cases}
\end{equation*}
for which $\csc\pi\leq K_0\leq\cot\pi$ and $\int_0^\pi K_0\,d\zeta=1$\,.

Let the transform $u=u^*\exp(\sigma z(\theta_t\omega))$, where the random variable $z(\theta_t\omega)$ is the stationary solution of  the Ornstein--Uhlenbeck \sde~(\ref{eq(22.4)}).
\begin{equation}\label{eq(22.4)}
dz+z\,dt=dW.
\end{equation}
Transform the \spde{}s~\eqref{1} to the following system of \spde{}s with random coefficients:
\begin{eqnarray}\label{2}
\D tu&=&au+\sigma z(\theta_t\omega)u-\exp(\sigma z(\theta_t\omega))uv,\nonumber \\
\D tv&=&-v+\DD xv+\sigma z(\theta_t\omega)v+\exp(\sigma z(\theta_t\omega))u^2-2\exp(2\sigma z(\theta_t\omega))\cK(u^2v)\,, \nonumber \\
\end{eqnarray}

If  $a=0$,  from Theorem~\ref{Thm(3.1)}, there exists  an infinite dimension  stochastic slow manifold.  Assume the stochastic slow manifold is $v=h^c(u,\omega)$, then from Theorem  \ref{k12}, the exact slow manifold satisfies
\begin{eqnarray*}&&
h^c_t(u,\omega)+h_u^c(u,\omega)[\sigma z(\theta_t\omega)u-\exp(\sigma z(\omega))uh^c(u,\omega)]
\\&&{}
+(h^c(u,\omega)-\DD xh^c(u,\omega))
-\sigma z(\theta_t\omega)h^c(u,\omega)-\exp(\sigma z(\theta_t\omega))u^2
\\&&{}
+2\exp(2\sigma z(\theta_t\omega))\cK(u^2h^c(u,\omega))=0,
\end{eqnarray*}
for $\sigma$ small enough.
From the construction of operator~$\cKo$, the above system of \spde{}s~\eqref{2} has a stochastic slow manifold of
\begin{eqnarray*}v=h(u,\omega)=\exp(\sigma z(\theta_t\omega))\cKo u^2.\end{eqnarray*}
By the equation \eqref{th2},  the evolution on the center manifold is governed by the \spde
\begin{equation*}
\D tu=\sigma z(\theta_t\omega)u-\exp(2\sigma z(\theta_t\omega))u\cKo u^2.
\end{equation*}


\paragraph{Acknowledgement} This work was supported by the Australian Research Council   grants DP0774311 and DP0988738, and by the NSF grant 1025422.

\end{document}